\documentclass{svproc}
\usepackage{url}
\usepackage{hyperref}
\usepackage{enumerate}
\usepackage{amsmath}
\usepackage{amssymb}
\usepackage{amsfonts}
\usepackage{xcolor}
\usepackage{ntheorem}

\def\N{\mathbb{N}}

\def\R{\mathbb{R}}

\def\S{\mathbb{S}}

\newcommand{\arccosh}{\operatorname{arccosh}}
\newcommand{\af}{\alpha}
\newcommand{\ep}{\varepsilon}

\begin{document}
\mainmatter              
\title{A geometric reduction method for some fully nonlinear first-order PDEs on semi-Riemannian manifolds}
\titlerunning{Geometric reduction method for fully nonlinear first order PDEs}  
%
\author{Juan Carlos Fernández\inst{1} \and Eddaly Guerra-Velasco\inst{2}
\and Oscar Palmas\inst{3} \and Boris A. Percino-Figueroa\inst{4}}

\authorrunning{J.C. Fernández, E. Guerra-Velasco, O. Palmas, B. Percino-Figueroa} 
%
%
\institute{Departamento de Matemáticas, Facultad de Ciencias, Universidad Nacional Autónoma de México\\
\email{jcfmor@ciencias.unam.mx}
\and
CONAHCYT-Facultad de Ciencias en Física y Matemáticas, Universidad Autónoma de Chiapas, México\\
\email{edaly.velasco@unach.mx, eguerra@conahcyt.mx}
\and
Departamento de Matemáticas, Facultad de Ciencias, Universidad Nacional Autónoma de México\\
\email{oscar.palmas@ciencias.unam.mx}
\and
Facultad de Ciencias en Física y Matemáticas, Universidad Autónoma de Chiapas, México\\
\email{boris.percino@unach.mx}}

\maketitle              

\begin{abstract}
Given a semi-Riemannian manifold $(M,\langle \cdot,\cdot\rangle_g),$ we use the transnormal functions defined on $M$ to reduce fully nonlinear first order PDEs of the form
\[
F(x,u,\langle \nabla_g u, \nabla_g u \rangle_g) = 0,\qquad \text{on }M
\]
into ODEs and obtain local existence results of solutions which are constant along the level sets of the transnormal functions. In particular, we apply this reduction method to obtain new solutions to eikonal equations with a prescribed geometry.

\bigskip

\end{abstract}

\section{Introduction}
Let $(M,\langle \cdot,\cdot\rangle_g)$ be a semi-Riemannian manifold without boundary of dimension $n\geq 3$. In this paper, we are interested in solutions to a fully nonlinear partial differential equation of the form 
\begin{equation}\label{Problem:MainEquation}
F(x, u(x),  \left\langle \nabla_g u(x),\nabla_g u(x) \right\rangle_g ) = 0,\qquad \text{on }M.
\end{equation}
where $u\colon M\to\R$ is a differentiable function. Notice that, when  $F$ is independent of $u$, \eqref{Problem:MainEquation} becomes
\begin{equation*}
 F(x, \left\langle \nabla_g u(x),\nabla_g u(x) \right\rangle_g ) = 0,\qquad \text{on }M.
\end{equation*}
which is a generalization of the eikonal equation 
\begin{equation}\label{Probelm:Eikonal}
\langle \nabla_{g(x)} u(x) , \nabla_{g(x)} u(x) \rangle_g=U(x),\quad U\in C^{\infty}(M).
\end{equation}


In this paper, we will prove the existence of local solutions to  problem \eqref{Problem:MainEquation} with a prescribed geometry on the level sets in two different ways. To establish our main results, we define and recall some concepts.

For a smooth function $f\colon M\to\R^k $, we will say that $F\colon  M\times \R^2\to \R $ is \emph{$f$-invariant in the first variable} if for any
$x_1,x_2\in M$ such that $f(x_1)=f(x_2)$, we have that 
\begin{equation*}
    F(x_1,r,p)=F(x_2,r,p),
\end{equation*}
for all  $(r,p)\in\R^2$. This is equivalent to saying that there is a unique smooth function $\widehat{F}\colon \mathrm{Im}\,f\times\R^2\subset \mathbb{R}^{k+2}\to\R$ such that for all $(r,p)\in\R^2$,
\begin{equation*}
F(x,r,p) = \widehat{F}(f(x),r,p),\qquad x\in M.
\end{equation*}

We say that a neighborhood $\Omega\subset M$ is $f$-invariant if $x\in\Omega$ implies $y\in\Omega$ for any $y\in M$ such that $f(x)=f(y)$. In other words, $\Omega$ is the union of level sets of $f$.

We make extensive use of transnormal functions. Given a semi-Riemannian manifold $(M,g)$, we say that $f\colon M\rightarrow \mathbb{R}$ is a \emph{transnormal} function if there exists a smooth function $a\colon\mathbb{R}\rightarrow\mathbb{R}$ such that
 \begin{equation*}
     \langle \nabla_g f, \nabla_g f \rangle_g = a\circ f.
 \end{equation*}
 
Examples of transnormal functions include isoparametric functions (see, for example,  \cite{FernandezPalmas2021,Hahn1984,LawnOrtega2022} for the semi-Riemannian setting and the book \cite{CecilRyanBook} in the Riemannian setting). In the Riemannian setting, there are explicit examples of transnormal functions which are not isoparametric (cf. \cite{Miyaoka2013}), but almost nothing is known in the semi-Riemannian case. 

Isoparametric functions have proven to be useful to find solutions for many PDEs on manifolds; for example, in nonlinear second order elliptic and ultrahyperbolic PDEs \cite{FernandezPetean2020,FernandezPalmas2021}, nonlinear higher order elliptic PDEs \cite{ClappFernandezSaldana2021,FernandezPalmasTorres2023,GrunauAhmedouReichel2008} and mean curvature flow equations \cite{LawnOrtega2022,OrtegaYildirim2024}. The main idea in this context is to use isoparametric functions to reduce the PDEs into ODEs which are simpler to analyze, generalizing the idea of looking for radial solutions in $\mathbb{R}^n$ or in a ball in $\mathbb{R}^n$. However, for many fully nonlinear first order PDEs, as equation \eqref{Problem:MainEquation}, it suffices to perform the reduction method by using only transnormal functions. 

In this paper we show two instances of this method, reducing problem \eqref{Problem:MainEquation} to an ODE or a two-dimensional PDE defined in a rectangle in $\mathbb{R}^2$. Let us describe briefly the method, main results and applications for each instance. In first place, for the reduction into an ODE, the key result is the following one.


\begin{proposition}\label{Proposition:OneDimensionalReduction}
Let $f\colon M\rightarrow \mathbb{R}$ be a transnormal function, $w\colon \mathbb{R}\rightarrow\mathbb{R}$ smooth and suppose $F\colon M\times\mathbb R^2\to\mathbb R$ is $f$-invariant in the first variable. Then $u=w\circ f$ is a solution to the equation
\begin{equation*}
F(x, u, \langle \nabla_g u, \nabla_g u \rangle_g) = 0,\qquad \text{on } M
\end{equation*}
if and only if $w$ is a solution to
\begin{equation}\label{Problem:MainOneDimensionalEquation}
\widehat{F}\left(t, w(t),a(t)\vert w'(t)\vert^2\right) = 0,\quad \text{in }\mathrm{Im}\,f\subset\mathbb{R}.
\end{equation}
\end{proposition}

 

The proof of this result is a straightforward computation. As a consequence, we obtain our first main result, a  local existence result for equation \eqref{Problem:MainEquation}:

\begin{theorem}\label{Theorem:OneDimensionalReduction}
Let $(M,g)$ be a semi-Riemannian manifold, $f\colon M\rightarrow\mathbb{R}$ be a transnormal function with
$\langle \nabla_g f, \nabla_g f \rangle_g = a\circ f$ and let $t_0\in\mathbb{R}$ be a regular value of $f$. Define $M_0:=f^{-1}(t_0).$ Then, for any $C^1$ function $F\colon M\times\mathbb{R}^2\rightarrow\mathbb{R}$ which is $f$-invariant in the first variable, and for any $(x_0,r_0,p_0)\in M_0\times \mathbb{R}\times(\mathbb{R}\smallsetminus\{0\})$ satisfying  
\begin{equation}\label{Equation:Theorem:ImplicitFunctionTheoremHypothesis}
F(x_0,r_0,p_0) = 0,\quad  \frac{\partial F}{\partial p} (x_0,r_0,p_0)\neq 0 \quad \text{and}\quad a(t_0)p_0>0,
\end{equation}
equation \eqref{Problem:MainEquation} posseses a unique $f$-invariant solution defined in an $f$-invariant domain $\Omega\subset M$ such that $M_0\subset\Omega$.
\end{theorem}

We prove this result in Section \ref{Section:OneDimensional}, which is an easy consequence of the Theorem of Existence and Uniqueness of ODEs. In the same section, and in order to illustrate the method behind the proof, we give some applications  in case of the eikonal equation \eqref{Probelm:Eikonal} in semi-Riemannian and Riemannian manifolds. In particular, using isoparametric functions defined in the semi-Euclidean space and in the pseudo-sphere \cite{FernandezPalmas2021,Hahn1984}, we construct, in Examples \ref{Example:ToyExampleEuclidian} and \ref{Example:CartanMunznerEikonal}, several solutions to this problem in both manifolds, with prescribed nontrivial geometry on the level sets. In Example \ref{Example:GammaInvariantEikonal}, we do the same reduction for cohomogeneity-one Riemannian manifolds \cite{FernandezPalmasTorres2023}. Moreover, in Example \ref{Example:HyperbolicOneDimensional}, we give a concrete and interesting example of a solution to equation \eqref{Probelm:Eikonal} defined in the hyperbolic Riemannian space, which is constant along the level sets of a transnormal function having non-isoparametric level sets \cite{Miyaoka2013}. In particular, we have the following result for the Riemannian eikonal equation. We denote by $\mathrm{Isom}(M,g)$ the group of isometries of $M$.

\begin{corollary}\label{Corollary:Eikonal}
Let $c>0$ and consider the eikonal equation 
\begin{equation}\label{Equation:RiemannianEikonal}
\langle \nabla_g u, \nabla_g u \rangle_g = c, \qquad \text{on }M,
\end{equation}
where $(M,g)$ is Riemannian.
\begin{enumerate}[(i)]
\item If $(M,g)$ is a closed Riemannian manifold and $\Gamma$ is a closed subgroup of $\mathrm{Isom}(M,g)$ acting by cohomogeneity one on $M$, then there exists a $\Gamma$-invariant solution to \eqref{Equation:RiemannianEikonal},  defined in $M\smallsetminus \mathcal{Z}$, where $\mathcal{Z}$ is the union of the singular orbits of the action.
\item If $(M,g)$ is the sphere with its canonical metric $(\mathbb{S}^n,g_0)$, then, for any isoparametric hypersurface $M_0\subset \mathbb{S}^n$, there exists a solution to equation \eqref{Equation:RiemannianEikonal}, constant on $M_0$ and defined in $\mathbb{S}^n\smallsetminus \mathcal{Z}$, where $\mathcal{Z}$ is a union of two submanifolds of $M$ of codimension $\geq 2$.
\item If $(M,g)$ is the hyperbolic space with its canonical metric $(\mathbb{H}^n,g_0)$, there exists a non isoparametric hypersurface $M_0$ in $\mathbb{H}^n$ and a globally defined solution $u:\mathbb{H}^n\rightarrow\mathbb{R}$ such that $u$ is constant on $M_0$.
\end{enumerate}
\end{corollary}

We prove this Corollary at the end of Section \ref{Section:OneDimensional}.

\medskip

In Section \ref{Section:TwoDimensional} we will explore the local existence of solutions to problem \eqref{Problem:MainEquation} with a different qualitative behavior. To describe it, let $ (L,g_L) $ and $ (N,g_N)$ be two connected and complete Riemannian manifolds, without boundary, of dimensions $s$ and $k$, respectively, and consider the semi-Riemannian warped product
\begin{equation}\label{Eq:WarpedProduct}
L\times_\af N:=(L\times N,g:=-g_L+\af^2g_N),
\end{equation}
where $\af\colon L\to\R$ is smooth and positive.

Let $f_L\colon L\to\R$ and $f_N\colon N\to\R$ be transnormal functions and consider a $C^2$ function $F\colon (L\times N)\times \R^2\to\R$ which is $(f_L\times f_N)$-invariant in the first variable, that is,
\begin{equation*}
F(x_1,z_1,r,\tau)=F(x_2,z_2,r,\tau),
\end{equation*}
whenever $f_L(x_1)=f_L(x_2)$ and $f_N(z_1)=f_N(z_2)$. As before, this implies the existence of a smooth function $\widehat{F}\colon \R^2\times\R^2\to \R$ such that
\begin{equation}\label{eq:2-inv}
F(x,z,r,\tau)=\widehat{F}(f_L(x),f_N(z),r,\tau). 
\end{equation}
In this setting, problem \eqref{Problem:MainEquation} can be written as:
\begin{equation}\label{Proble:2partial}
    F(x,y,u,\langle \nabla_g u,\nabla_g u\rangle_g)=0\qquad \text{on }\  L\times_\alpha N.
\end{equation}

In Section \ref{Section:TwoDimensional} we will show the local existence of functions $u\in C^{\infty}(L\times N)$ satisfying \eqref{Proble:2partial} and such that, for fixed $x\in L$ and $y\in N$, the functions $u(\cdot,y)$ and $u(x,\cdot)$ are constant on the level sets of $f_L$ and $f_N$, respectively. The key step is to reduce the problem \eqref{Proble:2partial} into a $2$-dimensional PDE of the first order in a rectangle in $\mathbb{R}^2$, easier to handle and visualize, for instance, by numerical methods.


We state the reduction method in the following proposition. In what follows, we define the intervals $I_{f_L}:=\mathrm{Im}\,f_L$ and $I_{f_N}:=\mathrm{Im}\,f_N$, the parameters $t:=f_L(x)\in I_{f_N}$ and $s:=f_N(y)\in I_{f_N}$, so that $\widehat{F}=\widehat{F}(t,s,r,\tau)$ is defined in $I_{f_L}\times I_{f_N}\times\mathbb{R}^2$, and the functions $a_{f_L}$ and $a_{f_N}$ are such that $\vert \nabla_{g_L}f_L\vert^2_g = a_{f_L}\circ f$ and $\vert \nabla_{g_N}f_N\vert^2_g = a_{f_N}\circ f$.

\begin{proposition}\label{Proposition:TwoDimensional}
Let $F$ be $(f_L\times f_N)$-invariant in the first variable as in \eqref{eq:2-inv}, and let $w\colon \mathbb{R}^2\rightarrow\mathbb{R}$ be smooth. If $\alpha=\widehat{\alpha}\circ f_L$ with $\widehat{\alpha}\colon \mathbb{R}\rightarrow\mathbb{R}$ smooth, then  $u=w(f_L,f_N)$ is a solution of \eqref{Proble:2partial} if and only if $w$ satisfies
\begin{equation}\label{Problem:Main2Dimensions}
\widehat{F}\left(t,s,w(t,s), - a_{f_L}(t) w_t^2(t,s) + \frac{a_{f_N}(s)}{\widehat{\alpha}^2(t)} w_s^2(t,s) \right) = 0,\quad \text{in } I_{f_L}\times I_{f_N},
\end{equation}
where $w_t:=\frac{\partial w}{\partial t}$ and $w_s:=\frac{\partial w}{\partial s}$.
\end{proposition}

We prove this proposition in Section \ref{Subsection:Computations}. As before, this proposition also implies a local existence result of $(f_L\times f_N)$-invariant solutions. In order to state it, define the singular Riemannian foliations in $L$ and $N$,  induced by the transnormal functions $f_L$ and $f_N$, 
\begin{equation*}
\mathcal{F}_{f_L}:=\{f_L^{-1}(t): t\in\mathbb{R}\}\qquad\text{and}\qquad \mathcal{F}_{f_N}:=\{f_N^{-1}(s) : s\in\mathbb{R}\},
\end{equation*}
respectively. Since $(L,g_L)$ and $(N, g_N)$ are Riemannian manifolds, the leaves of both foliations are also smooth manifolds (see \cite{Wang1987}). Notice also that $f_L$ and $f_N$ induce functions $C_{f_L}\colon\mathcal{F}_{f_L}\to \R$ and $C_{f_N}\colon\mathcal{F}_{f_N}\to \R$ given as $C_{f_L}(f_L^{-1}(t))=f_L(x)$ for any $x\in f_L^{-1}(t)$ and $C_{f_N}(f_N^{-1}(s))=f_N(z)$ for any $z\in f^{-1}(s)$. We say that a function $\Sigma_{f_L}\colon I\subset \mathbb{R}\rightarrow\mathcal{F}_{f_L}$ is a \emph{smooth uniparametric family of leaves in $L$} if $C_{f_L}\circ \Sigma_{f_L}\colon  I\subset\mathbb{R}\rightarrow\mathbb{R}$ is smooth. In the same way, we define a smooth uniparametric family of leaves in $N$.
As a consequence, we have the following local existence result.

\begin{theorem}\label{Theorem:TwoDimensional}
Let $f_L$ and $f_N$ be transnormal in $L$ and $N$, and let $t_0$ and $s_0$ be regular values of $f_L$ and $f_N$, respectively. Define $L_0:=f_L^{-1}(t_0)$ and $N_0:=f_N^{-1}(s_0).$ Let $F\colon (L\times N)\times\mathbb{R}^2\rightarrow\mathbb{R}$ be a $C^2$, $(f_L\times f_N)$-invariant function on the first variable. Suppose  that $(x_0,z_0,r_0,p_0,q_0)\in L_0\times N_0\times\mathbb{R}^3$ and 
\begin{equation*}
\tau_0:=-a_{f_L}(t_0)p_0^2+\frac{a_{f_N}(s_0)}{\widehat{\alpha}^{2}(t_0)}q_0^2
\end{equation*}
are such that $(p_0,q_0)\neq(0,0)$,
\begin{equation}\label{Equation:Theorem:LocalSolvability1}
F(x_0,z_0,r_0,\tau_0)=0,
\end{equation}
and suppose further that
\begin{equation}\label{Equation:Theorem:NontrivialCondition}
\frac{\partial F}{\partial \tau}(x,z,r,\tau)\neq 0
\end{equation}
in a neighborhood of $(t_0,s_0,r_0,\tau_0)$. Then, for any given smooth function \linebreak$R\colon (-\varepsilon,\varepsilon)\rightarrow\mathbb{R}$ and any pair of smooth one parameter families $\Sigma_{f_L},\Sigma_{f_N}$ defined in $(-\varepsilon,\varepsilon)$ such that 
\begin{enumerate}[(i)]
\item[\emph{($\Sigma.1$)}]\qquad $\Sigma_{f_L}(0)= L_0\ $ and $\ \Sigma_{f_N}(0)= N_0$,
\item[\emph{($\Sigma.2$)}] \label{Equation:Theorem:LocalSolvability2}
\qquad$q_0\;\frac{d}{d\zeta}\big\vert_{\zeta=0}\left(\Sigma_{f_L}\circ C_{f_L}\right) \neq \ p_0\;\frac{d}{d\zeta} \big\vert_{\zeta=0}\left(\Sigma_{f_N}\circ C_{f_N} \right)$, and
\item[\emph{($\Sigma.3$)}]
\label{Equation:Theorem:LocalSolvability3}
\qquad $R'(0) = p_0\; \frac{d}{d\zeta}\big\vert_{\zeta=0}\left(\Sigma_{f_L}\circ C_{f_L}\right)  +  q_0\; \frac{d}{d\zeta}\big\vert_{\zeta=0}\left(\Sigma_{f_N}\circ C_{f_N} \right)$,
\end{enumerate}
the problem \eqref{Proble:2partial} has an $(f_L\times f_N)$-invariant solution $u$ defined in an $(f_L\times f_N)$-invariant open neighborhood of $L_0\times N_0$, satisfying the condition
\begin{equation*}
u(\Sigma_{f_L}(\zeta),\Sigma_{f_N}(\zeta)) = R(\zeta),\quad \zeta\in (-\varepsilon,\varepsilon).
\end{equation*}
\end{theorem}

We prove this Theorem in Section \ref{Subsection:DistanceFunctions}. In the same section, we give a concrete example in the case of the de Sitter space in order to sketch the idea behind the method.

Our paper is organized as follows. In Section \ref{Section:OneDimensional} we prove Theorem \ref{Theorem:OneDimensionalReduction}, give  concrete examples where the one dimensional reduction method holds true and prove Corollary \ref{Corollary:Eikonal}. In Section \ref{Subsection:Computations}, we compute several identities involving norms of gradients of functions defined in warped products of the form \eqref{Eq:WarpedProduct} and prove Proposition \ref{Proposition:TwoDimensional}. Finally, in Section \ref{Subsection:DistanceFunctions}, we prove Theorem \ref{Theorem:TwoDimensional} and give a concrete example of the two-dimensional reduction method. In this same section, in order to simplify equation \eqref{Problem:Main2Dimensions}, we briefly introduce the theory of distance functions induced by transnormal functions on Riemannian manifolds.

\section{The one dimensional reduction method}\label{Section:OneDimensional}

In this section, we prove Theorem \ref{Theorem:OneDimensionalReduction} and apply it in several examples.

\bigskip
\noindent\emph{Proof of Theorem \ref{Theorem:OneDimensionalReduction}.}
Let $(x_0,r_0,p_0)\in M_0\times\mathbb{R}\times\left(\mathbb{R}\smallsetminus\{0\}\right)$ be a point satisfying the hypothesis \eqref{Equation:Theorem:ImplicitFunctionTheoremHypothesis}. Without loss of generality, suppose that $p_0>0$, so that $a(t_0)>0$. By Proposition \ref{Proposition:OneDimensionalReduction}, solving \eqref{Problem:MainEquation} is equivalent to solving \eqref{Problem:MainOneDimensionalEquation}. The hypotheses imply that
\begin{equation*}
\widehat{F}(t_0,r_0,p_0) = 0, \quad \text{and}\quad  \frac{\partial \widehat{F}}{\partial p} (t_0,r_0,p_0)\neq 0.
\end{equation*}
By the Implicit Function Theorem there exists an open set $U\subset\mathbb{R}^2$ and a $C^1$ function $H\colon U\rightarrow\mathbb{R}$ such that $H(t_0,r_0)=p_0$, $\widehat{F}(t,r,H(t,r))=0$ and $H(t,r)>0$ for any $(t,r)\in U$. 
Consider now the Cauchy problem
\begin{equation*}
w'(t) = \sqrt{\frac{H(t,w(t))}{a(t)}},\qquad w(t_0)=r_0.
\end{equation*}
As $H$ and $a$ have the same sign near $t_0$, and as $a(t_0)\neq 0$, by the Theorem of Existence and Uniqueness of solutions to initial value ODEs, there exists $\varepsilon>0$ and a unique $C^1$ solution $w\colon (t_0-\varepsilon,t_0+\varepsilon)\rightarrow\mathbb{R}$ to this problem.  Hence $a(t)w'(t)^2=H(t,w(t))$ and $w$ solves the problem \eqref{Problem:MainOneDimensionalEquation} in the same interval. By Proposition \ref{Proposition:OneDimensionalReduction}, the function $u:=w\circ f$ is $f$-invariant and satisfies equation \eqref{Problem:MainEquation} in $\Omega:=f^{-1}(t_0-\varepsilon,t_0+\varepsilon)$. By construction, $\Omega$ is an $f$-invariant open set and $M_0\subset\Omega$. \hfill $\square$
\bigskip

In what follows, we develop some examples to sketch the reduction method and see how the local existence holds true.

\begin{example}\label{Example:ToyExampleEuclidian}
Let $M=\mathbb{R}^n_s$ be the semi-Euclidean space of signature $s<n$. For any $A\in\mathrm{Sym}(\mathbb{R}^n_s)$, $a\in\mathbb{R}^n_s$ and $\alpha\in\mathbb{R}$ such that $(A-\alpha I)A=0$, it is shown in \cite[Proposition 2.3]{Hahn1984} that the function
\begin{equation*}
f\colon \mathbb{R}^m_s\rightarrow\mathbb{R},\quad f(x) = \langle Ax,x\rangle + 2\langle a,x \rangle
\end{equation*}
is transnormal, in fact isoparametric, and satisfies that
\begin{equation*}
\langle \nabla f, \nabla f \rangle = 4\alpha f + \beta,
\end{equation*}
where $\beta=4\langle a, a\rangle$.
Consider the eikonal equation \eqref{Probelm:Eikonal} and suppose $U$ is $f$-invariant. Let  $\widehat{U}\colon \mathbb{R}\rightarrow\mathbb{R}$ be a smooth function satisfying $U=\widehat{U}\circ f$. Hence, \eqref{Probelm:Eikonal} reduces to
\begin{equation*}
(4\alpha t+ \beta) \vert w'(t)\vert^2 = \widehat{U}(t),\quad t\in \mathrm{Im}\,f,
\end{equation*}
where $\mathrm{Im}\,f$ depends on the parameters $\alpha$ and $\beta$; see, for instance \cite[Appendix A]{FernandezPalmas2021}. If $\alpha\neq0$, this equation has a unique singularity at $t_\ast:=-\frac{\beta}{4\alpha}$. Note that, for any $t\ne t^*$, $4\alpha t + \beta$ and $\widehat{U}(t)$ must have the same sign. Hence, for any $t\neq t_\ast$, we can write
\begin{equation*}
w'(t) =\pm \sqrt{\frac{\widehat{U}(t)}{4\alpha t + \beta}}.
\end{equation*}
For instance, if $t_0<t_\ast$, we have an explicit expression for $w$ given as follows:
\begin{equation*}
w(t) = \pm \int_{t_0}^t \sqrt{\frac{\widehat{U}(s)}{4\alpha s + \beta}}\  ds,\qquad t\in (-\infty,t_\ast).
\end{equation*}
Hence, we can explicitly define the function $u\colon f^{-1}(-\infty,t_\ast)\rightarrow\mathbb{R}$ by
$u(x)=w\circ f(x)$, which is $f$-invariant by definition. The open set $\Omega:=f^{-1}(-\infty,t_\ast)$ is just an $f$-invariant neighborhood of the regular level set $M_0:=f^{-1}(t_0)$.  \hfill $\square$
\end{example}

\begin{example} \label{Example:CartanMunznerEikonal}
Consider the pseudo-sphere $(\mathbb{S}^n_s,g_0)$ of dimension $n\geq 3$ and signature $0\leq s<n$, and let $f\colon\mathbb{S}^n_s\rightarrow\mathbb{R}$ be an isoparametric function defined by a Cartan-Münzner polynomial in $\mathbb{R}^n_s$, see \cite[Section 2.2]{FernandezPalmas2021}. Thus, $f$ is transnormal and satisfies
\begin{equation*}
\langle \nabla_{g_0} f, \nabla_{g_0} f \rangle_{g_0} = a\circ f,\quad a(t)=\ell^2(1-t^2),
\end{equation*}
where $\ell\in\{1,2,3,4,6\}$ are the number of distinct principal curvatures of the hypersurfaces induced by $f$. In this case, $\mathrm{Im}\,f$ is $[-1,1]$ in the Riemannian case ($s=0$) and could be either $\mathbb{R}$, $(-\infty,-1],(-\infty,1],[-1,\infty)$ or $[1,\infty)$ in the semi-Riemannian case ($s>0$), see \cite[Appendix B]{FernandezPalmas2021}. Let $U\in C^\infty(\mathbb{S}_s^n)$ be $f$-invariant, so that $U=\widehat{U}\circ f$ and suppose further that $\widehat{U}$ is positive in $(-1,1)$ and negative in $\mathbb{R}\smallsetminus[-1,1]$. Then, by Proposition \ref{Proposition:OneDimensionalReduction}, a function $u=w\circ f$ is a solution to  the eikonal equation \eqref{Probelm:Eikonal} in $\mathbb{S}_s^n$ if and only if $w$ satisfies the equation
\begin{equation}\label{Equation:EikonalSphere}
\ell^2(1-t^2)\vert w'(t)\vert^2 = \widehat{U}(t),\qquad t\in\mathrm{Im}\,f.
\end{equation}
Observe that this equation has one or two singularities at $t=\pm 1$,  corresponding to the focal manifolds $M_\pm := f^{-1}(\pm 1)$, where, possibly, one is empty,  depending on $\mathrm{Im}\,f$. We next solve separately the cases $t\in(-1,1)$ and $t\in\mathbb{R}\smallsetminus[-1,1]$.

For $t\in(-1,1)$, consider the change of variables $t=\cos s$ and $v(s):=w(t)$, $s\in(0,\pi)$; then $w$ solves equation \eqref{Equation:EikonalSphere} if and only if $v$ solves
\begin{equation*}
\vert v'(s) \vert^2 = \frac{1}{\ell}\widehat{U}(\cos s),\qquad s\in(0,\pi).
\end{equation*}
As $\widehat{U}>0$ in $[-1,1]$, we can solve explicitly for any $s_0,s\in(0,\pi)$ and obtain
\begin{equation*}
v(s) = \frac{1}{\ell}\int_{s_0}^s\sqrt{\widehat{U}(\cos(\zeta))} d\zeta.
\end{equation*}
Hence, $u=v\circ\arccos f$ solves \eqref{Probelm:Eikonal} in $\Omega:= f^{-1}(-1,1)$ and $M_0:= f^{-1}(t_0)\subset \Omega$, where $t_0=\arccos s_0$.

Next, observe that $w(t)$ is a solution to \eqref{Equation:EikonalSphere} for $t\in(1,\infty)$ if and only if $\widetilde{w}(t)=w(-t)$ is a solution to the same equation for $t\in(-\infty,1)$. So, it suffices to solve equation \eqref{Equation:EikonalSphere} for $t>1$. To do this, consider the change of variables $t=\cosh s$ and let $v(s):=w(t)$, $s\in(0,\infty)$. Then $w$ solves \eqref{Equation:EikonalSphere} in $(1,\infty)$ if and only if $v$ solves
\begin{equation*}
\vert v'(s)\vert^2 = -\frac{1}{\ell^2}\widehat{U}(\cosh s),\qquad s\in(0,\infty).
\end{equation*}
As $\widehat{U}<0$ in $(1,\infty)$, this equation can be solved explicitly for any $s_0,s\in(0,\infty)$, giving that
\begin{equation*}
v(s) = \frac{1}{\ell}\int_{s_0}^s\sqrt{-\widehat{U}(\cosh \zeta)}\;d\zeta.
\end{equation*}
This yields the $f$-invariant solution $u:=v\circ\arccosh f$, defined in $\Omega:= f^{-1}(1,\infty)$, which is constant on the hypersurface $M_0=f^{-1}(t_0)$, with $t_0:=\arccosh s_0$. \hfill$\square$
\end{example}

\begin{example}\label{Example:GammaInvariantEikonal}
Let $(M,g)$ be a closed Riemannian manifold and let $\Gamma$ be a closed subgroup of $\mathrm{Isom}(M,g)$. By the theory of cohomogeneity one actions (Cf. \cite[Section 4]{FernandezPalmasTorres2023}), there are exactly two singular orbits $\Gamma x_1$ and $\Gamma x_2$, for some $x_1,x_2\in M$. If  $\mathrm{dist}_g$ denotes the geodesic distance in $M$, the distance function $d_\Gamma:M\rightarrow\mathbb{R}$ given by
\begin{equation*}
d_\Gamma(x):=\mathrm{dist}_g(x,\Gamma x_1)
\end{equation*}
is constant along each $\Gamma$-orbit and satisfies
\begin{equation*}
\langle \nabla_g d_\Gamma, \nabla_g d_\Gamma\rangle_g = 1,\quad \text{in } M\smallsetminus\mathcal{Z},\quad \text{and}\quad \mathrm{Im}\,d_\Gamma = [0,\delta_\Gamma]
\end{equation*}
with $\mathcal{Z}:=\Gamma x_1 \cup \Gamma x_2$ and $\delta_\Gamma := \mathrm{dist}_g(\Gamma x_1,\Gamma x_2)$. Hence, for any smooth and positive $\Gamma$-invariant function $U:M\rightarrow\mathbb{R}$, by Proposition \ref{Proposition:OneDimensionalReduction}, a $\Gamma$-invariant function $u:M\smallsetminus\mathcal{Z}\rightarrow\mathbb{R}$ of the form $u=w\circ d_\Gamma$ is a solution to the eikonal equation 
\begin{equation}\label{Equation:EikonalGammaInvariant}
\vert \nabla_g u\vert_g = U, \qquad \text{on }M\smallsetminus\mathcal{Z}
\end{equation}
if and only if $w$ is a solution to the problem
\begin{equation*}
\vert w'(t)\vert^2 = \widehat{U}(t) \qquad \text{in }(0,\delta_\Gamma).
\end{equation*}
where $\widehat{U}:[0,\delta_\Gamma]\rightarrow\mathbb{R}$ is smooth and such that $U=\widehat{U}\circ d_{\gamma}$. As $U>0$, $\widehat{U}$ is also positive and for any $t_0,t\in(0,\delta_\Gamma)$, the solution to this problem is given explicitly by
\begin{equation*}
w(t):=\int_{t_0}^t \sqrt{\widehat{U}(\zeta)}\;d\zeta.
\end{equation*}
Thus, the function $u:=w\circ d_\Gamma$ is a $\Gamma$-invariant function satisfying equation \eqref{Equation:EikonalGammaInvariant} in $d_\Gamma^{-1}(0,\delta_\Gamma)=M\smallsetminus\mathcal{Z}$. \hfill $\square$
\end{example}

\begin{example} \label{Example:HyperbolicOneDimensional}
Consider the Riemannian hyperbolic space $\mathbb{H}^n$ with its canonical metric $g_0$. Following Miyaoka's example given in \cite[Theorem 1.5]{Miyaoka2013}, we construct a transnormal, non-isoparametric function with foliation $\mathcal{F}_{f}$. To this end, take any totally geodesic $\mathbb{H}^{n-1}$ in $\mathbb{H}^n$ and let $M_0$ be a slight deformation of $\mathbb{H}^{n-1}$ in such a way that the principal curvatures $\lambda_i\colon M_0\rightarrow\mathbb{R}$ are not constant and $\vert \lambda_i\vert<1$, $i=1,\ldots,n-1.$ Then, there exists a transnormal function $f\colon \mathbb{H}^{n-1}\rightarrow\mathbb{R}$ such that $M_0\in\mathcal{F}_f$, $\mathrm{Im}\,f =\mathbb{R}$, all the leaves of $\mathbb{F}_f$ are regular level sets of $f$ and 
\begin{equation*}
\vert \nabla_{g_0} f(x)\vert_{g_0} = 1\qquad \text{for every }x\in\mathbb{H}^n.
\end{equation*}
Hence, as in Example \ref{Example:GammaInvariantEikonal}, if $M_0=f^{-1}(t_0)$ for some $t_0\in\mathbb{R}$, the function $u\colon\mathbb{H}^n\rightarrow\mathbb{R}$ given by $u:= w\circ f$, where
\begin{equation*}
w(t):=\int_{t_0}^t\sqrt{\widehat{U}(\zeta)} d\zeta 
\end{equation*}
is a globally defined $f$-invariant solution to the eikonal equation 
\begin{equation*}
\vert \nabla_{g_0}u\vert_{g_0} = U\qquad \text{in }\mathbb{H}^n,
\end{equation*}
for any $f$-invariant function $U=\widehat{U}\circ f$, with $\widehat{U}\colon\mathbb{R}\rightarrow\mathbb{R}$ positive. In particular, $u$ is constant on $M_0$ and $M_0$ is not isoparametric, because isoparametric hypersurfaces in $\mathbb{H}^n$ have constant principal curvatures, see, for instance \cite[Corollary 3.7]{CecilRyanBook}. \hfill $\square$
\end{example}
\bigskip 

\noindent\emph{Proof of Corollary \ref{Corollary:Eikonal}.}
Just fix $s=0$ in Example \ref{Example:CartanMunznerEikonal} and take $U\equiv c> 0$ in Examples \ref{Example:CartanMunznerEikonal}, \ref{Example:GammaInvariantEikonal}  and \ref{Example:HyperbolicOneDimensional}. \hfill$\square$

\section{The two dimensional reduction method}\label{Section:TwoDimensional}

\subsection{Some computations on warped products}\label{Subsection:Computations}

In this section, we prove Proposition \ref{Proposition:TwoDimensional}, beginning with some elementary computations for the warped product manifold $(L^s\times_\alpha N^k, g= -g_L + \alpha^2 g_N).$ First, notice that in local coordinates, the metric $g$ and its inverse have the following expressions
\begin{equation}\label{coef:met}
 		 g_{ij}=\begin{cases}
 		-(g_L)_{ij}, & 1\leq i,j\leq s,\\
 		\alpha^2(g_N)_{ij}, & s+1\leq i,j\leq s+k,\\
 		0, &\text{otherwise},
 	\end{cases} 
 	\end{equation}
and 
 \begin{equation}\label{coef:metinv}
 	g^{ij}=\begin{cases}
 		-g_L^{ij}, & 1\leq i,j\leq s,\\
 		\alpha^{-2}g_N^{ij}, & s+1\leq i,j\leq s+k,\\
 		0, &\text{otherwise},
 	\end{cases} 
 \end{equation}
respectively. Now, recall that the gradient of a function $u$ is given in coordinates $y=(x,z)$ as
\begin{equation*}
\nabla_g u=\sum_j\left(\sum_k g^{jk}\partial_{y_k}u\right)\partial_{y_j}.
\end{equation*}

 


 
We have the following computations for the gradient and its interior products.

  \begin{lemma}
  	For any $u\in C^{\infty}(L\times N)$ we have that
   \begin{align*}
       \nabla_g u&=-\nabla_{g_L} u+\af^{-2}\nabla_{g_N} u 
       , \\
  		\langle \nabla_g u,\nabla_g u\rangle_g &=-|\nabla_{g_L} u|^2_{g_L}+\af^{-2}|\nabla_{g_N} u|^2_{g_N}.
   \end{align*}
  \end{lemma}
  \begin{proof}
  	For a point $y=(x,z) $ in a coordinate chart, from \eqref{coef:met} and \eqref{coef:metinv}, we have that
  	\begin{align*}
  		\nabla_g u&=\sum_j^{s+k}\left(\sum_i^{s+k}g^{ji}\partial_{y_i}u\right)\partial_{y_j}\\
&=\sum_{j=1}^s\left(\sum_{i=1}^{s+k}g^{ji}\partial_{y_i}u\right)\partial_{x_j}+
\sum_{j=s+1}^{s+k}\left(\sum_{i=1}^{s+k}g^{ji}\partial_{y_i}u\right)\partial_{z_j}\\
  	 	 &=\sum_{j=1}^s\left(\sum_{i=1}^{s}-g_L^{ji}\partial_{x_i}u\right)\partial_{x_j}+
  	 	 \sum_{j=s+1}^{s+k}\left(\sum_{i=s+1}^{s+k}\af^{-2}g_N^{ji}\partial_{z_i}u\right)\partial_{z_j}\\
  	 	 &=-\nabla_{g_L}u+\af^{-2}\nabla_{g_N}u.
  	\end{align*}
  Since $\nabla_{g_L}u\in \Gamma(TL)$ and $\nabla_{g_N}u\in \Gamma(TN)$, if $\nabla_{g_L}u=\sum_{i=1}^sa_i\partial_{x_i}$ and $\nabla_{g_N}u=\sum_{i=s+1}^{s+k}b_i\partial_{z_i}$, for some smooth coefficients $a_i$ and $b_i$, then
  \begin{align*}
  	\langle\nabla_g u,&\nabla_g u\rangle_g=\langle -\nabla_{g_L}u+\af^{-2}\nabla_{g_N}u,-\nabla_{g_L}u+\af^{-2}\nabla_{g_N}u  \rangle\\
  	&=\langle \nabla_{g_L}u,\nabla_{g_L}u\rangle_g+2\langle-\nabla_{g_L}u,\af^{-2}\nabla_{g_N}u \rangle +\langle\af^{-2}\nabla_{g_N}u,\af^{-2}\nabla_{g_N}u \rangle\\
&=\left\langle\sum_{i=1}^sa_i\partial_{x_i},\sum_{i=1}^sa_i\partial_{x_i}\right\rangle_g+\af^{-4}\left\langle\sum_{i=s+1}^{s+k}b_i\partial_{z_i},\sum_{i=s+1}^{s+k}b_i\partial_{z_i} \right\rangle_g\\
  	&=\sum_{i,j=1}^s a_ia_j\langle{\partial }_{x_i},{\partial}_{x_j}\rangle_g+\af^{-4}\sum_{i,j=s+1}^{s+k}b_ib_j\langle{\partial }_{z_i},{\partial}_{z_j}\rangle_g\\
  &	=-|\nabla_{g_L}u|^2+\af^{-2}|\nabla_{g_N}u|^2.
  \end{align*}\qed
  \end{proof}
  


For later use, we compute the norms of the gradients of $(f_L\times f_N)$-invariant functions, not necessarily transnormal. In what follows, $t:=f_L(x), s:=f_N(z).$

  \begin{lemma}
  Let $w\in C^1(\mathbb{R}^2)$ and let $f_L\colon L\rightarrow \mathbb{R}$ and $f_N\colon N\rightarrow\mathbb{R}$ be smooth. If $u:=w(f_L,f_N)$, then
  		\begin{equation*}  			|\nabla_{g_L}u|^2_{g_L}=w_t^2(f_L,f_N)|\nabla_{g_L}f_L|^2_{g_L},\quad\text{and}\quad |\nabla_{g_N}u|^2_{g_N}=w_{s}^2(f_L,f_N)|\nabla_{g_L}f_L|^2_{g_L}.
  		\end{equation*} 
  \end{lemma}
  
 \begin{proof}
 	In local coordinates $y=(x,z)$, the differential of $u=w( f_L,f_N)$ is given by 
\begin{align*}
	 d(u)&=d(w\circ (f_L,f_N)) =dw(f_L,f_N)\circ d(f_L\times f_N)\\
	 	 &=dw\circ (df_L\times df_N) \\
	 	 &=\big(w_t(f_L,f_N),w_s(f_L,f_N)\big)\begin{bmatrix}
	 	 	\nabla_{g_L}f_L&0\\0&\nabla_{g_N}f_N
	 	 \end{bmatrix}\\
 	 	 &=\big(w_t(f_L,f_N)\nabla_{g_L}f_L,w_{s}(f_L,f_N)\nabla_{g_N}f_N\big).
\end{align*}
Hence
\begin{equation*} \partial_{x_i}u =w_t(f_L,f_N)\partial_{x_i}f_L,
\quad \text{and} \quad	 \partial_{z_i}u =w_s(f_L,f_N)\partial_{z_i}f_N.
\end{equation*}
Therefore, 
\begin{align*}
	\nabla_{g_L}(w(f_L,f_N))&=\sum_{j=1}^{s}\left(\sum_{i=1}^{s}g_L^{ij}\partial_{x_i}(w(f_L,f_N))\right)\partial_{x_j}\\
	&=\sum_{j=1}^{s}\left(\sum_{i=1}^{s}g_L^{ij}w_t(f_L,f_N)\partial_{x_i}f_L\right)\partial_{x_j}\\
	&=w_t(f_L,f_N)\sum_{j=1}^{s}\left(\sum_{i=1}^{s}g_L^{ij}\partial_{x_i}f_L\right)\partial_{x_j}\\
	&=w_t(f_L,f_N)\nabla_{g_L}f_L,
\end{align*}
where we conclude that
\begin{align*}
	|\nabla_{g_L}u|^2&=\langle w_t(f_L,f_N)\nabla_{g_L}f_L,w_t(f_L,f_N)\nabla_{g_L}f_L\rangle_{g_L}\\
	&=w_t^2(f_L,f_N)|\nabla_{g_L}f_L|^2_{g_L}.
\end{align*}

Similar computations give
\begin{equation*}
\nabla_{g_N}u=w_s(f_L,f_N)\nabla_{g_N}f_N \quad\text{ and }\quad	|\nabla_{g_L}u|^2=w_s^2(f_L,f_N)|\nabla_{g_L}f_L|^2_{g_L},
\end{equation*}
as we wanted.
\qed
 \end{proof}

Now suppose that $f_L\colon L\to\R$ and $f_N\colon N\to\R$ are transnormal functions satisfying
\begin{equation*}
 |\nabla_{g_L}f_L|_{g_L}^2 =a_L\circ f_L \quad\text{and}\quad
 |\nabla_{g_N}f_N|_{g_N}^2 =a_N\circ f_N,
\end{equation*}
for some smooth functions $a_L,a_N\colon \R\to\R$. We have the following immediate consequence of the previous two lemmas.

\begin{corollary}\label{Corollary:GradientTransnormal}
	For $f_L,f_N$ transnormal, $u=(f_L,f_N)$ and $\af=\widehat{\alpha}\circ f_L$,
	\begin{equation*}
		\langle \nabla_g u,\nabla_g u\rangle_g=-w_t^2(f_L,f_N)a_L\circ f_L+\frac{a_N\circ f_N}{\widehat{\alpha}^2\circ f_L}w_{s}^2(f_L,f_N).
	\end{equation*}
\end{corollary}

\bigskip

\noindent\emph{Proof of Proposition \ref{Proposition:TwoDimensional}}.
Using the identity in Corollary \ref{Corollary:GradientTransnormal}, the proof is straightforward.\hfill$\square$

\subsection{Reduction using distance functions}\label{Subsection:DistanceFunctions}

In this subsection, we simplify   
\eqref{Problem:Main2Dimensions} via distance functions, and prove Theorem \ref{Theorem:TwoDimensional}.  By the theory of transnormal functions for Riemannian manifolds (see \cite{Miyaoka2013} and \cite{Wang1987}), the connected components of the critical sets of $f_L$ and $f_N$ are submanifolds of $L$ and $N$ respectively, called \emph{focal submanifolds}. There are, at most, two focal submanifolds for a given transnormal function, and, in some cases, this set is empty. In what follows, we will suppose further that $f_L$ and $f_N$ admit, at least, one focal manifold, say $L_0$ and $N_0$, and we will denote by $L_1$ and $N_1$ the second, possibly empty, focal submanifold. In this case, we can define the distance functions $d_L\colon L\rightarrow\mathbb{R}$ and $d_N\colon N\rightarrow\mathbb{R}$ given by
\begin{equation*}
d_L(x):=\mathrm{dist}_{g_L}(x,L_0)\quad\text{and}\quad d_N(y):=\mathrm{dist}_{g_N}(y,N_0).
\end{equation*}
Let $\delta_L,\delta_N\in(0,\infty]$ be the supremum of these functions in $L$ and $N$, respectively. The case of infinite distance corresponds to the non-compact case of $L$ or $N$, while, in the case of compact manifolds, $\delta_L=\mathrm{dist}_{g_L}(L_1,L_0)$ and $\delta_N=\mathrm{dist}_{g_N}(N_1,N_0)$. Moreover, for any $x\in L\smallsetminus(L_0\cup L_1)$ and any $z\in N\smallsetminus(N_0\cup N_1)$, the distance functions satisfy
\begin{equation*}
\langle \nabla_{g_L} d_L(x) , \nabla_{g_L} d_L (x) \rangle_{g_L} = 1 \quad\text{and}\quad \langle \nabla_{g_N} d_N(z) , \nabla_{g_N} d_N (z) \rangle_{g_N} = 1
\end{equation*}
and, so, they are transnormal in $ L\smallsetminus(L_0\cup L_1)$ and in $N\smallsetminus(N_0\cup N_1)$ respectively, with $a_{d_L}=1$ and $a_{d_N}=1$, and $I_{d_L}=[0,\delta_L]$ and $I_{d_N}=[0,\delta_N]$.

Now, observe that $d_L$ and $f_L$ have the same level sets and, hence, $f_L$-invariant functions are $d_L$-invariant functions and vice-versa. The same is true for $d_N$ and $f_N$. In fact, $d_L$ and $d_N$ are just reparametrizations of $f_L$ and $f_N$. For this reason, and without loss of generality, we will consider directly that $F$ is $(d_L\times d_N)$-invariant in its first variable and that $\alpha$ is $d_L$-invariant. Moreover, since $\mathcal{F}_{f_L}=\mathcal{F}_{d_L}$ and $\mathcal{F}_{f_N}=\mathcal{F}_{d_N}$, the one parameter families of hypersurfaces satisfy that $\Sigma_{f_L}=\Sigma_{d_L}$ and $\Sigma_{f_N}=\Sigma_{d_N}$.
\bigskip

\noindent\emph{Proof of Theorem \ref{Theorem:TwoDimensional}.} By the previous remarks, it suffices to prove the result in the $d_L$ and $d_N$-invariant setting, that is, suppose that $F$ is $(d_L\times d_N)$-invariant in the first variable, and let $\widehat{F}\colon I_{d_L}\times I_{d_N}\times\mathbb{R}^2\rightarrow\mathbb{R}$ such that $F=\widehat{F}(d_L\times d_N,\cdot,\cdot)$ in $L\times N\times\mathbb{R}^2$. In the same way, as $\alpha$ is $f_L$-invariant, it is $d_L$ invariant and $\alpha = \widehat{\alpha}\circ d_L$ in $L\smallsetminus(L_0\cup L_1)$. Accordingly, we consider the one parameter families $\Sigma_{d_L}$ and $\Sigma_{d_N}$ such that $\Sigma_{d_L}(0)=L_0=d_L^{-1}(t_0)$ and $\Sigma_{d_N}(0)=N_0=d_L^{-1}(s_0)$, for regular values $t_0\in (0,\delta_L)$ and $s_0\in(0,\delta_N)$ of $d_L$ and $d_N$, respectively. To simplify notation, let $T:=C_{d_L}\circ\Sigma_{d_L}$ and $S:=C_{d_N}\circ\Sigma_{d_N}$. 

Since $d_L$ and $d_N$ are transnormal in $L\smallsetminus(L_0\cup L_1)$ and $N\smallsetminus (N_0\cup N_1)$, respectively, by Proposition \ref{Proposition:TwoDimensional}, to solve equation \eqref{Problem:MainEquation} with $(d_L\times d_N)$-invariance, it suffices to find $w\in C^1\left((0,\delta_L)\times (0,\delta_N)\right)$ solving the reduced equation
\begin{equation*}
\widehat{F}\left(t,s,w(t,s), -w_t^2(t,s) + \frac{1}{\widehat{\alpha}^2(t)} w_s^2(t,s) \right) = 0,\quad \emph{\text{ in }}\ (0,\delta_L)\times(0,\delta_N).
\end{equation*}
Observe this equation can be written as a first order fully nonlinear PDE of the form
\begin{equation*}
H(t,s,w,w_t,w_s) = 0 \qquad\text{ in }(0,\delta_L)\times(0,\delta_N),
\end{equation*}
where $H\colon (0,\delta_L)\times(0,\delta_N)\times\mathbb{R}^3\to\mathbb R$ is given by
\begin{equation*}
H(t,s,r,p,q)=\widehat{F}\left( t,s,r, -p^2 + \widehat{\alpha}^{-2}(t) q^2 \right).
\end{equation*}
 As 
\begin{equation*}
\frac{\partial H}{\partial p} (t,s,r,p,q) = 2p \frac{\partial \widehat{F}}{\partial \tau}(t,s,r,-p^2+\widehat{\alpha}^{-2}q^2)
\end{equation*}
and
\begin{equation*}
\frac{\partial H}{\partial q} (t,s,r,p,q) = 2q \frac{\partial \widehat{F}}{\partial \tau}(t,s,r,-p^2+\widehat{\alpha}^{-2}q^2),
\end{equation*}
then, for any $(x_0,z_0)\in L_0\times N_0$, by the hypotheses \eqref{Equation:Theorem:LocalSolvability1}, \eqref{Equation:Theorem:NontrivialCondition}, \emph{\hyperref[Equation:Theorem:LocalSolvability2]{$(\Sigma.2)$ }} and \emph{\hyperref[Equation:Theorem:LocalSolvability3]{$(\Sigma.3)$ }}, we have that
\begin{equation*}
H(t_0,s_0,r_0,p_0,q_0) = 0, \qquad p_0T'(0) + q_0 S'(0) = R'(0),
\end{equation*}
that
\begin{equation*}
T'(0) \frac{\partial H}{\partial p} (t_0,s_0,r_0,p_0,q_0) - S'(0) \frac{\partial H}{\partial q} (t_0,s_0,r_0,p_0,q_0) \neq 0,
\end{equation*}
and that 
\begin{equation*}
\left( \frac{\partial H}{\partial p} \right)^2 (t,s,r,p,q) + \left( \frac{\partial H}{\partial q} \right)^2 (t,s,r,p,q) \neq 0
\end{equation*}
in a small neighborhood of $(t_0,s_0,r_0,p_0,q_0)$.  Hence, by the standard theory of first-order fully nonlinear equations (see, for instance, \cite[Section 4.6]{SalsaBook}), the Cauchy problem
\begin{equation*}
\begin{cases}
\widehat{F}\left( t,s,w(t,s),w_t^2(t,s) + \widehat{\alpha}^2w_{s}^2(t,s) \right)=0,\\
w(T(\zeta),S(\zeta))=R(\zeta),
\end{cases}
\end{equation*}
has a unique solution $w$ of class $C^1$ defined in a small neighborhood of $(t_0,s_0)$, and again, using Proposition \ref{Proposition:TwoDimensional}, the function
$u:=w(d_L,d_N)$ is $(d_L\times d_N)$-invariant and satisfies equation \eqref{Proble:2partial} in a small $(f_L\times f_N)$-invariant neighborhood of $L_0\times N_0$, as we wanted to prove. \hfill $\square$
\bigskip

To conclude this paper, let us analyze a simple application of the method in case of the de Sitter space. 

\begin{example}
 Let $n\geq 3$ and let $(\S^n,g_0)$ be the sphere with the usual metric and consider the de Sitter space
 \begin{equation*}
(M,g)=(\R\times \S^n, dt^2+ \cosh^2 t\, g_0).
 \end{equation*}
 For $n_1,n_2\in\N$ such that $n_1+n_2=n+1,\, n_1,n_2\geq2$, take 
\begin{equation*}
\S^n\subset\R^{n+1}=\R^{n_1}\times\R^{n_2}
\end{equation*}
and for $z=(z_1,z_2)\in\S^n$ define the function
\begin{equation*}
f(z)=\arccos(|z_1|^2-|z_2|^2).
\end{equation*}
This function is transnormal and satisfies that $\mathrm{Im}\, f=[0,\pi]$. The sets
\begin{equation*}
f^{-1}(0)=\S^{n_1-1}\times \{0\}\quad\text{and}\quad f^{-1}(\pi)=\{0\}\times\S^{n_2-1}
\end{equation*}
are the focal submanifolds, while 
\begin{equation*}
f^{-1}(s)=\S^{n_1-1}_{\cos s}\times\S^{n_2-1}_{\sin s}
\end{equation*}
are regular hypersurfaces for $s\in(0,\pi)$ and 
\begin{equation*}
|\nabla_{g_0}f|^2_{g_0}=4, \quad \text{ in }\S^n\smallsetminus f^{-1}\{0,\pi\}.
\end{equation*}

For any $(t_0,z_0)\in\mathbb{R}\times\left(\S^n\smallsetminus f^{-1}\{0,\pi\}\right),$ and any smooth function $\widehat{U}\colon\R\times[0,\pi]\to\R$, we will show next that the eikonal equation on $M$
\begin{equation}\label{eikon:examp}
    \langle \nabla_g u,\nabla_g u\rangle=\widehat{U}(t,z),
\end{equation}
has a solution which is $f$-invariant in the $z$-variable and is defined in a neighborhood of $\{t_0\}\times\S_{\cos s_0}^{n_1-1}\times\S_{\sin s_0}^{n_2-1}$ for suitable one-parameter families $\Sigma_{\R}$ and $\Sigma_{\S^n}$,  where $s_0=f(z_0)$.

 First, notice  that $\mathrm{Id}\colon\R\to\R$ is transnormal, since $\vert \nabla\mathrm{Id}(t)\vert=1$ for any $t\in\R$. So the function defining the eikonal equation \eqref{eikon:examp} is just
    \begin{equation*}
    F(t,z,r,\tau)=\tau-U(t,z).
    \end{equation*}
    This function is $(\mathrm{Id}\times f)$-invariant on the $(t,z)$-variables. Therefore, by Proposition \ref{Proposition:TwoDimensional}, any function of the form $u=w\circ(\mathrm{Id}\times f)$ solving \eqref{eikon:examp} must satisfy the equation 
\begin{equation}\label{equation:eik:reduced}
-w_t^2(t,s)+4\cosh^{-2}(t)w_s^2(t,s)=\widehat{U}(t,s)
    \end{equation}
    in $\R\times [0,\pi]$
    which can be written as 
\begin{equation*}
\widehat{F}\left(t,s,w(t,s),-w_t^2(t,s))+4\cosh^{-2}(t)w_s^2(t,s)\right)=0
\end{equation*}
where $\widehat{F}(t,s,w,\tau)=\tau-\widehat{U}(t,s)$.
    We proceed to solve this equation with the theory of fully nonlinear equations; cf. \cite{SalsaBook}.

    Fix $(t_0,z_0)\in \R\times\S^n\smallsetminus f^{-1}(\{0\})$ and let $s_0=f(z_0).$ Take $(p_0,q_0)\in\R^2\smallsetminus\{(0,0)\}$ such that
    \begin{equation}\label{eq:hiper}
        -p_0^2+4\cosh^{-2}(t_0)q_0^2-\widehat{U}(t_0,s_0)=0.
    \end{equation}
    Notice that such $(p_0,q_0)\neq(0,0)$ exists because \eqref{eq:hiper} defines a hyperbola in the $(p_0,q_0)$-variables.

    Then, the function $H\colon\R\times [0,\pi]\times \R^3\to\R$ given by 
    \begin{align*}
         H(t,s,r,p,q)&=\widehat{F}\left(t,s,r,-p^2+4\cosh^{-2}(t)q^2\right) \\
                    &=-p^2+4\cosh^{-2}(t) q^2-\widehat{U}(t,s)
    \end{align*}
   is used to rewrite \eqref{equation:eik:reduced} as 
   \begin{equation*}
   H(t,s,w(t,s),w_t(t,s),w_s(t,s))=0.
   \end{equation*}
   So, for any $r_0\in\R$, the point $(t_0,z_0,r_0,p_0,q_0)\in\R\times\left(\S^n\smallsetminus f^{-1}(\{0,\pi\})\right)\times \R^3$ is such that $F(t_0,z_0,\tau_0)=0,$ where 
   $\tau_0=-p_0^2+4\cosh^{-2}(t_0) q_0^2,$ and, therefore
   \begin{equation*}
   H(t_0,s_0,r_0,p_0,q_0)=0.
   \end{equation*}
    Since $\frac{\partial F}{\partial \tau}\equiv1\neq0,$
    \begin{equation*}
    H_p(t,s,r,p,q)=-2p,\quad H_q(t,s,r,p,q)=8\cosh^{-2}(t) q,
    \end{equation*}
    and the nontriviality hypothesis
    \begin{equation*}
    H_p^2+H_q^2=4p^2+64\cosh^{-2}(t) q^2
    \end{equation*}
    is satisfied for any $(p,q)\neq(0,0)$.

    For $\ep>0$, let $T,S\colon(-\ep,\ep)\to\R^3$ be such that 
    \begin{equation*}
  (T(\zeta),S(\zeta))\in\R\times(0,\pi),\quad \zeta\in(-\varepsilon,\varepsilon),
    \end{equation*}
    with 
    \[
    T(0)=t_0,\qquad S(0)=s_0\quad \text{ and }\quad T'(0)q_0,S'(0)p_0>0.
    \] 
    Then, for any smooth function $R\colon(-\ep, \ep)\to\R$ satisfying 
    \begin{equation*}
    R'(0)=p_0T'(0)+q_0S'(0),
    \end{equation*}
    the initial value problem
    \begin{equation*}
    \begin{cases}
     H(t,s,w(t,s),w_t(t.s),w_s(t,s))=0,\\
         w(T(\zeta),S(\zeta))=R(\zeta),
    \end{cases}
    \end{equation*}
    admits a $C^2$ solution in a neighborhood of $(t_0,s_0).$
    Defining $\Sigma_{\mathrm{Id}}(\zeta):=\{T(\zeta)\}$ and $\Sigma_{f}(\zeta):=f^{-1}(S(\zeta))\,$ we have that 
    \begin{equation*}
    \Sigma_{\mathrm{Id}}(\zeta)\circ C_{\mathrm{Id}}=T,\quad \Sigma_{f}(\zeta)\circ C_{f}=S,
    \end{equation*}
    and $u=w\circ(\mathrm{Id}\times f)$ is a $(\mathrm{Id}\times f)$-invariant solution defined in a neighborhood of $\{t_0\}\times\S_{\cos s_0}^{n_1-1}\times\S_{\sin s_0}^{n_2-1}$
    of the form 
    \begin{equation*}
    \Omega =(t_0-\ep_1,t_0+\ep_1)\times f^{-1}(s_0-\ep_2,s_0+\ep_2),
    \end{equation*}
    for some $\ep_1,\ep_2>0$ small enough. Moreover, $u$ satisfies that
    \begin{equation*}
    u(\Sigma_{Id}(\zeta),\Sigma_{f}(\zeta))=R(\zeta)\quad \text{ in } (-\ep_3,\ep_3),
    \end{equation*}
    for some $\ep_3>0$ small enough. \hfill$\square$
\end{example} 
\smallskip
\begin{remark}
More examples where this method applies are given by considering warped products of Riemannian space forms of the form
$L\times_{\alpha} N$, where $L$ is either the Euclidean space $\mathbb{R}^s$ or the hyperbolic space  $\mathbb{H}^s$, and $N$ may be chosen as $\mathbb{R}^k$, $\mathbb{H}^k$ or the sphere $\mathbb{S}^k$. As mentioned before, the case of the hyperbolic space is of particular interest due to the existence of transnormal functions having non isoparametric foliations \cite{Miyaoka2013}, see Example \ref{Example:HyperbolicOneDimensional} above.\hfill$\square$
\end{remark}

\bibliographystyle{plainurl}

\bibliography{References}

\end{document}